\documentclass{amsart}
\usepackage{amsfonts,dsfont}
\usepackage{amsmath}
\usepackage{amssymb}
\usepackage{fullpage}
\usepackage{color}
\usepackage{nccmath} 
\usepackage[shortlabels]{enumitem} 

\newtheorem{theorem}{Theorem}
\newtheorem{corollary}[theorem]{Corollary}
\newtheorem{lemma}[theorem]{Lemma}
\newtheorem{proposition}[theorem]{Proposition}

\newtheorem{remark}[theorem]{Remark}

\newtheorem{conjecture}[theorem]{Conjecture}

\theoremstyle{definition}

\numberwithin{equation}{section}
\numberwithin{theorem}{section}

\newcommand{\pma}{\begin{pmatrix}}
\newcommand{\epma}{\end{pmatrix}}

\newcommand{\enum}{\begin{enumerate}[{\rm (1)},leftmargin=*,itemsep=1ex]}
\newcommand{\eenum}{\end{enumerate}}

\DeclareMathOperator{\I}{Id}
\DeclareMathOperator{\V}{var}

\DeclareMathOperator{\spn}{span}

\DeclareMathOperator{\ad}{ad}

\DeclareMathOperator{\Der}{Der}

\DeclareMathOperator{\E}{End}

\begin{document}

\title{Differential codimensions and exponential growth}

\author[C. Rizzo]{Carla Rizzo}
\address{CMUC, Departamento de Matem\'atica, Universidade de Coimbra, 3004-501 Coimbra, Portugal}
\email{carlarizzo@mat.uc.pt}



\thanks{This work was supported by the Centre for Mathematics of the University of Coimbra - UIDB/00324/2020, funded by the Portuguese Government through FCT/MCTES}

\begin{abstract}
Let $A$ be a finite dimensional associative algebra with derivations over a field of characteristic zero, i.e., an algebra  whose structure is enriched by the action of a Lie algebra $L$ by derivations, and let $c_n^L(A),$ $n\geq 1,$ be its differential codimension sequence.
Such sequence is exponentially bounded
and  $\exp^L(A) = \lim_{n\to \infty}\sqrt[n]{c_n^L(A)}$ is an integer that can be computed, called differential PI-exponent of $A$.

In this paper we prove that for any Lie algebra $L$, $\exp^L(A)$ coincides with $\exp(A)$, the ordinary PI-exponent of $A$. 
Furthermore, in case $L$ is a solvable Lie algebra, we apply such result to classify varieties of $L$-algebras of almost polynomial growth, i.e., varieties of
exponential growth such that any proper subvariety has polynomial
growth.
\end{abstract}

\maketitle

\section{Introduction}

Let $A$ be an associative algebra over a field $F$ of characteristic zero and $\I(A)$ its $T$-ideal of polynomial identities.
One of the most interesting and challenging problems in combinatorial theory of polynomials identities is that of finding numerical invariants allowing to give a quantitative description of $\I(A)$. In this setting a very useful and important invariant is the sequence of codimensions $c_n(A)$, $n\geq 1$, of $A$, introduced by Regev in 1972.  More precisely, if $P_n$ is the vector space of multilinear polynomials in the non-commutative $n$ variables, $c_n(A) = \dim P_n/(P_n \cap \I(A))$ is called the $n$-th codimension of $A$. When the base field is of characteristic zero $\I(A)$ is determined by the multilinear
polynomials it contains, then the codimension sequence gives in some sense a quantitative measure of the identities satisfied by $A$. Regardless of its
importance, the exact computation of the codimensions of an algebra is an hard task
and it has been done for very few algebras. That is why one is led to study the asymptotic
behaviour of the codimensions. In this sense, Regev in \cite{Regev1972} showed that if $A$ satisfies a non-trivial
polynomial identity (PI-algebra for short), then the codimension sequence is exponentially bounded. 
Later Kemer \cite{Kemer1979} showed that such codimensions are either polynomially bounded or grow exponentially. Moreover, in \cite{GiambrunoZaicev1998} and \cite{GiambrunoZaicev1999} Giambruno and Zaicev  proved the Amitsur's conjecture for PI-algebras, i.e., they showed that for any PI-algebra $A$ over a field of characteristic zero the sequence
$(c_n(A))^{1/n}$ converges, and its limit is always an integer, called the exponent of $A$ and denoted by $\exp(A)$. Since then, extensive research on the
exponent of PI-algebras has been conducted.

Here we are interested in the growth of the differential identities of algebras, i.e., polynomial identities of algebras with an action of a Lie algebra by derivations. Recall that if $L$ a Lie $F$-algebra acting on $A$ by derivations, then such  action can be naturally extended to the action of the universal enveloping algebra $U(L)$ of $L$ and in this case $A$ is called algebra with derivations or $L$-algebra. Then one can define in a natural way the differential identities of $A$, i.e., polynomials in the variables $x^d$, $d\in U(L)$, vanishing on $A$. Such identities were introduced by Kharchenko in \cite{Kharchenko1978} (see also \cite{Kharchenko1979}) and in later years, after the paper \cite{GordienkoKochetov2014} of Gordienko and Kochetov  the interest on them grew.

Similarly to the ordinary case, one can attach to an $L$-algebra $A$ the differential codimension sequence $c_n^L(A)$, $n\geq 1$. In \cite{Gordienko2013JA} Gordienko showed that in case $A$ is finite dimensional $L$-algebra, $c_n^L(A)$ is exponentially bounded and he captured this exponential rate of growth answering positively to the Amitsur’s conjecture for this kind of algebras. More precisely, he proved that the limit $\lim_{n\to \infty} \sqrt[n]{c_n^L(A)} =
\exp^L(A)$ exists and is a non-negative integer called differential exponent, or $L$-exponent, of $A$ and he gave an explicit way to compute it.  As a consequence of \cite{Gordienko2013JA}, it turns out that the differential codimensions of a finite dimensional algebra are either polynomially bounded or grow exponentially (no intermediate growth is allowed).

The theory of differential identities is a natural generalization of the theory of ordinary polynomial identities arising when the Lie algebra $L$ acts trivially on $A$ and, as consequence,  $U(L)$ coincides with $F$. So,
at this point a question arise naturally: can we compare the differential exponent and the ordinary one of a given $L$-algebra?

Since $c_n(A)\leq c_n^L(A)$ for all $n \geq 1$,  clearly we have that $\exp(A)\leq \exp^L(A)$ and in \cite{GordienkoKochetov2014} Gordienko and Kochetov conjectured the following.
\begin{conjecture}
	If $A$ is a  finite dimensional $L$-algebra, then
	$$
	\exp(A)=\exp^L(A).
	$$
\end{conjecture} 
In the same paper they proved it in case $L$ is a finite dimensional semisimple Lie algebra and in \cite{RizzodosSantosVieira2021} the authors proved that $\exp^L(A)=1$ if and only if $\exp(A)=1$ where $L$ is any Lie algebra. 

In this paper we give a positive answer to the Gordienko-Kochetov's conjecture for any Lie algebra $L$.
It is important to highlight that this conjecture is no longer true if we consider other type of algebras with additional structure such for example
algebras with involution \cite{GiambMish2001CA} or algebras graded by an abelian finite group \cite{Valenti2011}.

 One of the main advantages of the exponent
is to have an integral scale allowing us to measure the growth of any non-trivial variety of algebras. So it becomes important to study
varieties with the same exponent and to determine those with the most distinguished properties.
In this setting  a celebrated theorem of Kemer characterizes varieties generated by an algebra of exponent less or equal to one as follows. If $G$ is the infinite dimensional Grassmann algebra over $F$ and $UT_2$ is the algebra of $2\times 2$ upper triangular matrices over $F$, then $\exp(A)\leq 1$ if and only if $G$, $UT_2$ do not belong to the variety $ \mathcal{V}$ generated by $A$. 

Now, if $\mathcal{V}$ is a variety of algebras, the growth of $\mathcal{V}$ is the growth of the
sequence of codimensions of a generating algebra. 
Hence the varieties generated by $G$ and $UT_2$ are the only of almost polynomial growth, i.e.,  they grow
exponentially but any proper subvariety grows polynomially. Similar results were also proved in the setting of varieties of graded algebras \cite{GiambMishZai,Valenti2011}, algebras with involution \cite{GiambMish}, algebras with superinvolution \cite{GiambrunoIoppoloLaMattina2016} and algebras with pseudoinvolution \cite{IoppoloMartino2018}.

Clearly $UT_2$ generates also $L$-variety, i.e, variety of $L$-algebras, of almost polynomial growth if we suppose that $L$ acts trivially on it. Another useful example of algebras with derivations generating an $L$-variety of almost polynomial growth is $UT_2^\varepsilon$, the $L$-algebra $UT_2$ with
$F\varepsilon$-action, where $\varepsilon$ is the inner derivation
induced by $e_{22}$ (see \cite{GiambrunoRizzo2018}). In \cite{Rizzo2021} the author proved that, up to $T_L$-equivalence, $UT_2$ and $UT_2^\varepsilon$ are the only $2\times 2$ upper triangular matrices generating $L$-varieties of almost polynomial growth and in \cite{MartinoRizzo2022} the authors completely classify all their subvarieties. Notice that for what concern the infinite dimensional Grassmann algebra so far we know just that $G$  generates  $L$-variety of almost polynomial growth if $L$ acts trivially on it and it does not generate a $L$-variety of almost polynomial growth if $L$ acts on it as a finite dimensional abelian Lie algebra (see \cite{Rizzo2020ART}).

Inspired by the above results
here we characterize $L$-varieties $\mathcal{V}$ having polynomial growth and we reach our goal in the setting of varieties generated
by finite dimensional $L$-algebras $A$ where $L$ is a solvable Lie algebra. In this setting we prove that $\mathcal{V}$ has polynomial growth
if and only if $UT_2, UT_2^\varepsilon \notin \mathcal{V}$. As a
consequence, there are only two varieties with
derivations generated by a finite dimensional algebra with almost polynomial growth.

\section{Preliminaries}
Throughout this paper $F$ will denote a field of characteristic zero, $A$ a finite dimensional associative $F$-algebra and $L$ a fixed Lie $F$-algebra. 

Recall that a \textit{derivation} of $A$ is a linear map $\delta:A\to A$ such that it satisfies the \textit{Leibniz rule}: for all $a,b\in A$
$$(ab)^\delta=a^\delta b + a b^\delta.$$ 
If $a\in A$, then the $F$-liner map $\ad_a:A\to A$ defined by $ x^{\ad_a}=[x,a]=xa-ax$ for all $x\in A$, is a derivation on $A$ called \textit{inner derivation} induced by the element $a$.  
Notice that the set of all derivation of $A$ is a Lie $F$-algebra denoted by $\Der(A)$ and the set $\ad (A)$ of all inner derivations of $A$ is a Lie ideal of $\Der(A)$. 
Throughout the paper, we will adopt the exponential notation for derivations, hence derivations will compose from left to right.

A Lie algebra $L$ acts on $A$ by derivation if there exists a homomorphism of Lie algebras $\varphi:L\to \Der(A)$. In particular, if $\bar{L}$ is a Lie subalgebra of $\Der(A)$, then we say that $L$ acts on $A$ as the Lie algebra $\bar{L}$ if $\varphi(L)=\bar{L}$. By the \textit{Poincaré--Birkhoff--Witt Theorem} the $L$-action on $A$ can be naturally extended to an $U(L)$-action, where $U(L)$ is the universal enveloping algebra of $L$ with product right-to-left (opposite to the usual one), in fact $\varphi$ can be naturally extended to an homomorphism of associative algebras $\phi:U(L)\to \E_F(A)$. In this way $A$ becomes a right $U(L)$-module and we call it \textit{algebra with derivations} or \textit{$L$-algebra}. Note also that by the \textit{Poincaré--Birkhoff--Witt Theorem}, if $\{\delta_i \ | \  i\in I\}$ is an ordered basis of $L$, then $U(L)$ has a basis $\{\delta_{i_1}\cdots \delta_{i_p} \ | \ i_1< \dots < i_p, \; i_k\in I, \; p\geq 0\}$. Thus $U(L)=U'(L)\oplus F\cdot1$ as vector spaces, where $U'(L)$ is the non-unital universal enveloping algebra of $L$ and $1=1_{U(L)}$ is the unit of $U(L)$.

Let $(A,\overline{L})$ and $(B, L')$ be two $L$-algebras, i.e., $L$ acts on $A$ as the Lie algebra $\overline{L}$ and on $B$ as the Lie algebra $L'$. An isomorphism of algebras $ \psi: A \to B$
is said to be a isomorphism of $L$-algebras if there exists an homomorphism of Lie algebra $\phi:\bar{L}\to L'$ such that $\psi(a^\delta) = \psi(a)^{\phi(\delta)}$, for any $a\in A$ and $\delta \in \overline{L}$ and in this case we write $A\cong_L B$. Notice that in case the two $L$-algebras $(A,\overline{L})$ and $(B, L')$ are isomorphic just as ordinary algebras we write $A\cong B$.

\smallskip

In order to define what a polynomial identity is for this kind algebras, we need to
introduce the free algebra with derivations. 
Given a basis $\mathcal{B}_{U(L)}=\{d_{i}\ | \  i\geq 0 \}$ of $U(L)$ and a countable set $X=\{x_1,x_2, \dots \}$, we let $F\langle X|L\rangle$ be the free associative algebra over $F$ with free formal generators $x_i^{d_j}$, $i >0,$ $ j\geq 0$ where we identify $x_i=x_i^{1}$, $1=d_0\in U(L)$. Moreover, for all $d=\sum_{i\in I}\alpha_i d_i\in U(L)$, where only a finite number of $\alpha_i\in F$ are non-zero, we set $x^d:= \sum_{i\geq 0}\alpha_i x^{d_i}$. Then $F\langle X |L\rangle$ has a structure of $L$-algebra by setting
$$
(x_{i_{1}}^{d_{j_{1}}} x_{i_{2}}^{d_{j_{2}}}\dots x_{i_{n}}^{d_{j_{n}}})^\delta =x_{i_{1}}^{ d_{j_{1}}\delta} x_{i_{2}}^{d_{j_{2}}}\dots x_{i_{n}}^{d_{j_{n}}}+ x_{i_{1}}^{ d_{j_{1}}} x_{i_{2}}^{ d_{j_{2}}\delta}\dots x_{i_{n}}^{d_{j_{n}}}+ \dots+x_{i_{1}}^{d_{j_{1}}} x_{i_{2}}^{d_{j_{2}}}\dots x_{i_{n}}^{ d_{j_{n}}\delta}
$$
for all $\delta\in L$ and $x_{i_{1}}^{d_{j_{1}}}x_{i_{2}}^{d_{j_{2}}}\dots x_{i_{n}}^{d_{j_{n}}}\in F\langle X|L\rangle$. Thus $F\langle X |L\rangle$ is called \textit{free algebra with derivations} or \textit{free $L$-algebra} and its elements are called \textit{differential polynomials} or \textit{$L$-polynomials}.  Note that our definition
of $F\langle X |L\rangle$  depends on the choice of the basis $\mathcal{B}_{U(L)}$ in $U(L)$. However such algebras can be identified in a natural way.

A differential polynomial $f(x_1, \dots , x_n)\in F\langle X| L\rangle$ is a differential identity, or an $L$-identity, of the $L$- algebra $A$, if $f(a_{1},\dots,a_{n})=0$ for any $a_{i}\in A$, and in this case we write $f\equiv 0$. We denote by $\I^L(A)$ the set of differential identities of $A$, which is a $T_L$-ideal of the free algebra with derivations, i.e., an ideal invariant under all endomorphisms $\psi$ of $F\langle X|L\rangle$ such that $\psi(f^d)=\psi(f)^d$ for all $f \in F\langle X| L\rangle$ and $d\in U(L)$. We shall use the following notation:  if $(A,\overline{L})$ is an $L$-algebra, i.e., there exists a surjective homomorphism of Lie algebra $\varphi:L\to \bar{L}\subseteq \Der(A)$, then  
in the set of generators of the $T_L$-ideal $\I^L(A)$ of $(A,\bar{L})$ we omit the differential identities of the type $x^\delta \equiv 0$ for all $\delta \in \ker\varphi$.

As in the ordinary case, in characteristic zero $\I^L(A)$ is completely determined by its multilinear polynomials. We denote by $P_n^L=\spn_F\{x_{\sigma(1)}^{d_{i_1}}\dots x_{\sigma(n)}^{d_{i_n}} \ | \ \sigma\in S_{n} ,d_{i_k}\in  \mathcal{B}_{U(L)} \}$ the vector space of multilinear differential polynomials in the variables $x_{1},\dots,x_{n}$, $ n\geq 1$ and
$$
c_n^L(A)=\dim_F \dfrac{P_n^L}{P_n^L \cap \I^L(A)}
$$
is called \textit{$n$th differential codimension}, or \textit{$L$-codimension}, of $A$. Notice that $c_n^L(A)$ is well defined because $U(L)$ acts on $A$ as a suitable subalgebra of $\E_F(A)$ and $A$ is a finite dimensional $L$-algebra. 

In order to capture the exponential rate of growth of the above sequence of codimensions, in \cite{Gordienko2013JA} the author
proved that for any finite dimensional $L$-algebra $A$, the limit
$$\exp^L(A)= \lim_{n\to \infty}\sqrt[n]{c_n^L(A)}$$
exists and is a non-negative integer, called \textit{differential exponent}, or \textit{$L$-exponent}, of $A$. As a consequence, if $A$ is a finite dimensional $L$-algebra, then the sequence of differential codimensions $c_n^{L}(A)$  growths exponentially or is polynomially bounded.

Recall that if $A$ is an $L$-algebra, then the variety of algebras with derivations generated by $A$ is denoted by $\V^{L}(A)$ and is called $L$-variety. The growth of $\mathcal{V}= \V^{L}(A)$ is the growth of the sequence $c_{n}^{L}(\mathcal{V})=c_{n}^{L}(A)$, $n\geq 1$. We say that the $L$-variety $\mathcal{V}$ has polynomial growth if $c_{n}^{L}(\mathcal{V})$ is polynomially bounded and $\mathcal{V}$ has almost polynomial growth if  $c_{n}^{L}(\mathcal{V})$ is not polynomially bounded but every proper $L$-subvariety of $\mathcal{V}$ has polynomial growth.

\smallskip

Notice that the theory of differential identities generalizes the ordinary theory of polynomial identities. In fact, any algebra $A$ can be regarded as $L$-algebra by letting $L$ act on $A$ trivially,
i.e., $L$ acts on $A$ as the trivial Lie algebra and $U(L)\cong F$.
Moreover, since $U(L)$ is an algebra with unit, we can identify in a natural way $P_n$ with a subspace of $P_n^L$. Hence $P_n\subseteq P^L_n$ and $P_n\cap \I(A)=P_n\cap \I^L(A)$. As a consequence, if for all $n\geq 1$ we denote by
$$c_n(A)=\dim_F \dfrac{P_{n}}{P_{n}\cap \I(A)},$$
the sequence of (ordinary) \textit{codimension} of $A$, then we have the following relation 
 \begin{equation}\label{relazione codimensioni}
 	c_n(A)\leq c_n^L(A), \quad \mbox{for all } n\geq 1.
 \end{equation}
Moreover, it is well known that if $A$ is an associative algebra (not necessarily finite dimensional) over a field of characteristic zero the limit
$$\exp(A)= \lim_{n\to \infty}\sqrt[n]{c_n(A)}$$
always exists and is a non-negative integer called the \textit{exponent} of $A$ (see \cite{GiambrunoZaicev1998,GiambrunoZaicev1999}).  By relation \eqref{relazione codimensioni} we have that if $A$ is a finite dimensional $L$-algebra, then $\exp(A)\leq\exp^L(A)$. In what follows we shall prove that actually $\exp(A)=\exp^L(A)$.

\section{Finite dimensional algebras with derivations and their codimensions}

In this section we shall prove the Gordienko-Kochetov's conjecture.
To this end we start by presenting the structure and properties of finite dimensional $L$-algebras.

Let $A$ be a finite dimensional $L$-algebra. If $B\subseteq A$, then we denote by $B^L$ the set all $b^d$ such that $b\in B$ and $d\in\mathcal{B}_{U(L)}$, where $\mathcal{B}_{U(L)}$ is a basis of $U(L)$. We say that  $B$ is $L$-invariant if $B^L\subseteq B$.
Thus we have the following definition.
 An ideal (subalgebra) $I$ of $A$ is an $L$-ideal (subalgebra) if it is $L$-invariant. 
 
 Let us denote by $J=J(A)$ the Jacobson radical of $A$. It is well known that $J(A)$ is an $L$-ideal of $A$ \cite[Theorem 4.2]{Hoch} and that $A$ is called semisimple if and only
 if $J(A) = 0$.

Since $A$ is  finite dimensional, then by Wedderburn-Malcev decomposition for associative algebras (see \cite[Theorem 3.4.3]{GiambrunoZaicev2005book}), we can find a maximal semisimple subalgebras $B\subseteq A$ such that
$$
	A=B+J.
$$
Moreover,
$B=B_1 \oplus\cdots \oplus  B_k$
where $B_1, \dots, B_k$ are simple algebras and are all minimal ideals of $B$.
In case $L$ is a semisimple Lie algebra, then such decomposition is $L$-invariant, i.e., we can find a maximal semisimple subalgebra $B$ that is $L$-invariant (\cite[Theorem 4]{GordienkoKochetov2014}). However if $L$ is not semisimple, although $J(A)$ is an $L$-ideal it may not exist an $L$-invariant Wedderburn-Malcev decomposition, i.e., it may happen that $B^L\nsubseteq B$ for every maximal semisimple subalgebra $B$ of $A$ (see \cite[Example 3]{RizzodosSantosVieira2021}).

\begin{lemma}\label{prop B_i^L}
	Let $A=B_1 \oplus \dots \oplus B_k + J$ a finite dimensional algebra where $B_1, \dots, B_k$ are simple algebras. If $A$ is an $L$-algebra, then 	$B_i^L \subseteq B_i +J$ for all $1\leq i \leq k$. Moreover, $B_i^L \subseteq J$ whenever $B_i\cong F$.
\end{lemma}
\begin{proof}
	Let  $\delta\in \Der(A)$. 
	By \cite[Theorem 4.3]{Hoch} $\delta = \ad_b + \ad_j +\delta'$, where $b\in B$, $j\in J$ and $\delta'\in \Der(A)$ is such that $a^{\delta'}=0$ for all $a \in B$. Let $1\leq i\leq k$ and $a\in B_i$. Then $a^\delta=a^{\ad_b}+ a^{\ad_j}$. Since $B_i B_r=0$, $r \neq i$,  $a^{\ad_b}\in B_i$. Thus since $J$ is an ideal of $A$, $a^\delta\in B_i +J$ and we have proved that $B_i^\delta\subseteq B_i +J$ for all $1\leq i \leq k$ and $\delta \in \Der(A)$. Now if $d\in U(L)$, by the Poincaré--Birkhoff--Witt Theorem we may assume that there exists $s\geq 0$ such that $d=\delta_1 \dots \delta_s$ where  $\delta_i\in L$ for all $1\leq i \leq s$. If $s=0$, then $d=1_{U(L)}$ and there is nothing to prove. So let us suppose that $s>0$. Since $B_i^\delta \subseteq B_i +J$ for any choose of $\delta \in L$ and $J$ is an $L$-ideal, then $B_i^L \subseteq B_i +J$.
	
	Now suppose that  $B_i\cong F$ for some $1\leq i \leq k$. Then $ B_i=\spn_F\{e_i\}$. Since $B_i B_r=0$, $r \neq i$, $e_i^\delta=e_i^{ \ad_b + \ad_j +\delta'}=e_i^{\ad_j}\in J$, where $\delta = \ad_b + \ad_j +\delta'\in \Der(A)$ with $b\in B$, $j\in J$ and $\delta'\in \Der(A)$ is such that $B^{\delta'}=0$. Thus since $J$ is an $L$-ideal, the proof is complete.
\end{proof}

Recall that an algebra $A$ is $L$-simple if $A^2\neq \{0\}$ and $A$ has no non-trivial $L$-ideals. Thus we have the following.

\begin{proposition}\label{Prop semisimple}
	Let $A$ be a finite dimensional $L$-algebra. Then
		\begin{itemize}
			\item[1)] If $A$ is semisimple, then $A=A_1 \oplus \dots \oplus A_m$ where $A_1, \dots , A_m$  are $L$-simple algebras and are all minimal $L$-ideals of $A$.
			
			\vspace{0.1mm}
			\item[2)] If $A$ is $L$-simple, then $A$ is simple.
		\end{itemize}
\end{proposition}

\begin{proof}
	Suppose that $A$ is semisimple. Then since $J(A)=0$, we have that $A=B_1\oplus \dots \oplus B_k$ where $B_1, \dots, B_k$ are simple algebras and are all minimal ideals of $A$. By Lemma \ref{prop B_i^L} it follows that $B_i^L\subseteq B_i$ for all $1\leq i \leq k$. Thus $B_1, \dots, B_k$ are $L$-invariant simple algebras and, as a consequence, $L$-simple. So the first statement is proved.
	As for the second, it is proved in \cite[Lemma 9]{GordienkoKochetov2014}.
\end{proof}

\begin{lemma}
	\label{A_1J...JA_k}
	Let $A=B_1 \oplus \dots \oplus B_k + J$ a finite dimensional algebra over an algebraically closed field $F$ of characteristic zero where $B_1, \dots, B_k$ are simple algebras  and let $A^+=A + F \cdot 1$. If $A$ is an $L$-algebra and $B_{1}^L A^+ B_2^L  \cdots B_{k-1}^L A^+  B_{k}^L\neq 0$, then  $B_1 J B_2 \cdots B_{k-1} J B_k \neq 0$
\end{lemma}
\begin{proof}
	If $J=0$ there is nothing to prove. So let $J\neq 0$. If $k=1$ again there is nothing to prove. So, assume that $k>1$. By hypothesis 
	\begin{equation}\label{condizione 1}
	b_1^{d_1}a_1b_2^{d_2}a_2\dots a_{k-1} b_k^{d_k}\neq 0
	\end{equation}
	 for some $b_i\in B_i$, $d_i\in \mathcal{B}_{U(L)}$ and $a_q\in A^+$ for $1\leq i \leq k$ and $1\leq q \leq k-1$. 
	 We claim that for all $2\leq i \leq k-1$ there exist $\bar{a}_{i-1}, \bar{a}_i \in A^+$ and $\bar{b}_i\in B_i$ such that  $b_{i-1}^{d_{i-1}} \bar{a}_{i-1}\bar{b}_i \bar{a}_i b_{i+1}^{d_{i+1}}\neq 0$.

	 By \eqref{condizione 1} for all $2\leq i \leq k-1$ we have that $b_{i-1}^{d_{i-1}} a_{i-1}b_i^{d_i} a_i b_{i+1}^{d_{i+1}}\neq 0$. If $b_i^{d_i}\in B_i$ the claim is proved. Then let us assume that $b_i^{d_i}\notin B_i$. Since $B_i$ is simple, it is a unitary algebra and we denote by $e_i$ its unit. Without loss generality we may suppose that $d_i=\delta_1 \dots \delta_r$ with $\delta_j\in L$, $1\leq j \leq r$, $r\geq 1$.  We proceed by induction on $r$.
	
	If $r=1$, then $d_i\in L$. By the Leibniz rule and since $e_i$ is the unit of $B_i$, we get that $b_i^{d_i}=e_i^{d_i}b_i + e_i b_i^{d_i}$. Thus it follows that
	$$b_{i-1}^{d_{i-1}} a_{i-1}e_i^{d_i}b_i  a_i b_{i+1}^{d_{i+1}} + b_{i-1}^{d_{i-1}} a_{i-1}e_i b_i^{d_i} a_i b_{i+1}^{d_{i+1}}\neq 0.$$
	 Now, if $b_{i-1}^{d_{i-1}} a_{i-1}e_i^{d_i}b_i  a_i b_{i+1}^{d_{i+1}}\neq 0$, by setting $\bar{a}_{i-1}=a_{i-1}e_{i}^{d_i}\in A^+$, $\bar{b}_i=b_i \in B_i$ and $\bar{a}_i=a_i\in A^+$ we get the desired conclusion. On the other hand, if  $b_{i-1}^{d_{i-1}} a_{i-1}e_i b_i^{d_i} a_i b_{i+1}^{d_{i+1}}\neq 0$, we get $b_{i-1}^{d_{i-1}} \bar{a}_{i-1}\bar{b}_i \bar{a}_i b_{i+1}^{d_{i+1}}\neq 0$ where $\bar{a}_{i-1}=a_{i-1}\in A^+$, $\bar{b}_i=e_i\in B_i$ and $\bar{a}_i= b_i^{d_i} a_i \in A^+$. So let us assume that $r>1$. Then again by Leibniz rule and since $e_i$ is the unit of $B_i$ we have that
	\begin{equation*}
	b_i^{d_i}=e_i^{d_i}b_i + e_i b_i^{d_i}+ \sum_{\mathcal{P}, \mathcal{Q}} e_i^{d_\mathcal{P}} b_i^{d_\mathcal{Q}}
	\end{equation*}
	where $\{\mathcal{P}, \mathcal{Q}\}$ is a partition of the set $\{1, \dots, r\}$ into two disjoint ordered non-empty subsets such that if $\mathcal{P}=\{p_1,\dots, p_s\}$ and $\mathcal{Q}=\{q_1, \dots, q_t\}$, then $d_\mathcal{P}= \delta_{p_1}\cdots \delta_{p_s}$ and $d_\mathcal{Q}=\delta_{q_1}\cdots \delta_{q_t}$. Since by hypothesis we have that  $b_{i-1}^{d_{i-1}} a_{i-1}b_i^{d_i} a_i b_{i+1}^{d_{i+1}}\neq 0$, then 
	\[ b_{i-1}^{d_{i-1}} a_{i-1} e_i^{d_i}b_i a_i b_{i+1}^{d_{i+1}}+ b_{i-1}^{d_{i-1}} a_{i-1} e_i b_i^{d_i} a_i b_{i+1}^{d_{i+1}}+ \sum_{\mathcal{P}, \mathcal{Q}} b_{i-1}^{d_{i-1}} a_{i-1} e_i^{d_\mathcal{P}} b_i^{d_\mathcal{Q}} a_i b_{i+1}^{d_{i+1}}\neq 0. \]
	If $b_{i-1}^{d_{i-1}} a_{i-1} e_i^{d_i}b_i a_i b_{i+1}^{d_{i+1}}\neq 0$ or $ b_{i-1}^{d_{i-1}} a_{i-1} e_i b_i^{d_i} a_i b_{i+1}^{d_{i+1}}\neq 0$ we are done. So, let suppose that there exists a partition $\{\mathcal{P}, \mathcal{Q}\}$ of $\{1, \dots, r\}$ such that  $b_{i-1}^{d_{i-1}} a_{i-1} e_i^{d_\mathcal{P}} b_i^{d_\mathcal{Q}} a_i b_{i+1}^{d_{i+1}}\neq 0$. Since $ a_{i-1} e_i^{d_\mathcal{P}} \in A^+$, by the inductive hypothesis the claim is proved. 
	
	We proved that there exist $\bar{b}_i\in B_i$, $2\leq i \leq k-1$, and $\bar{a}_1, \dots, \bar{a}_{k-1}\in A^+$ such that 
	$b_1^{d_1} \bar{a}_1 \bar{b}_2 \bar{a}_2\dots \bar{a}_{k-1}b_k^{d_k}\neq 0$. Now with a similar argument of above we can prove that there exist $\bar{b}_1\in B_1$, $\bar{b}_k \in B_k$ and $\tilde{a}_1, \tilde{a}_{k-1}\in A^+$ such that $\bar{b}_1\tilde{a}_1 \bar{b}_2 \neq 0$ and $\bar{b}_{k-1}\tilde{a}_{k-1}\bar{b}_k\neq 0$ and, as a consequence, we have that $\bar{b}_1 \tilde{a}_1 \bar{b_2} \bar{a}_2 \dots \tilde{a}_{k-1}\bar{b}_{k}\neq 0$.
	 Since $B_r B_s\neq 0$ for all $1\leq r,s \leq k$, $r \neq s$, it follows that $\bar{a}_1, \dots \bar{a}_{k-1}\in J$ and the lemma is proved.
\end{proof}

Next  we recall the  characterizations of the exponent $\exp(A)$ and  $L$-exponent $\exp^L(A)$ of $A$.

\begin{theorem}\label{exp}\cite[Section 6.2]{GiambrunoZaicev2005book} If $A$ is a finite dimensional algebra over a field of characteristic zero, then 
	\begin{equation*} \label{Characterization of exp}
	\exp(A)=\max\{\dim_F(B_{i_1}\oplus B_{i_2}\oplus\cdots\oplus B_{i_r}) \; | \;  B_{i_1}J B_{i_2}J \cdots J B_{i_r}\neq 0, \;  1\leq r \leq k, i_p\neq i_s, 1\leq p,s\leq r \},
	\end{equation*}
	where $A=B_1\oplus \cdots \oplus B_k+J$ with $B_1,\dots ,B_k$  simple algebras and $J=J(A)$ is the Jacobson radical of $A$.
\end{theorem}

\begin{theorem}\label{L-exp}\cite[Theorems 1 and 3]{Gordienko2013JA} Let $A$ be a finite dimensional $L$-algebra over a field of characteristic zero. 
	If $J=J(A)$ is the Jacobson radical of $A$ and $A/J=\overline{A}_1\oplus \cdots \oplus \overline{A}_m$ with $A_1, \dots, A_m$ $L$-simple algebras, then 
	\begin{equation*} \label{Characterization of L-exp}
		\exp^L(A)=\max\{\dim_F(\overline{A}_{i_1}\oplus \overline{A}_{i_2}\oplus\cdots\oplus \overline{A}_{i_r}) \; | \;  A_{i_1}^L A^+  A_{i_2}^L A^+ \cdots A^+  A_{i_r}^L\neq 0, \;  1\leq r \leq m, i_p\neq i_s, 1\leq p,s\leq r \},
	\end{equation*}
	where $A^+=A+ F\cdot 1$ and $A_{i_1}, \dots, A_{i_k}$ are a subalgebra of $A$ (not necessary $L$-invariant) such that $\pi(A_{i_r}) =\overline{A}_{i_r}$ for all $1\leq r \leq k$, where $\pi: A\to A/J$ is the canonical projection. 
\end{theorem}

\begin{remark} \label{rmk isomorphism B_i and A_i}
	Let $\pi:A \to A/J=\overline{A}$ be the canonical projection where $J=J(A)$ is the Jacobson radical of $A$. If $B$ is the maximal semisimple subalgebra of $A$ such that $A=B+J$ and 	
	$$
	B=B_1 \oplus\cdots \oplus  B_k
	$$
	where $B_1, \dots, B_k$ are simple algebras (not necessary $L$-invariant) and are all minimal ideals of $B$, then clearly $\pi_{|B}:B \to \overline{A}$ is an isomorphism of algebras and $\pi(B_1), \dots , \pi(B_k)$ are simple subalgebra of $\overline{A}$. 
	On the other hand, since $J$ is an $L$-ideal of $A$, $\overline{A}$ is a semisimple $L$-algebra and by \em{1)} of Proposition \ref{Prop semisimple} we have that
	\begin{equation*}\label{A/J decomposition}
	\overline{A}=\overline{A}_1\oplus \cdots \oplus \overline{A}_m
	\end{equation*}
	where $\overline{A}_1, \dots, \overline{A}_m$ are $L$-simple algebras and are all minimal $L$-ideals of $\overline{A}$. Then by \em{2)} of Proposition \ref{Prop semisimple}, $\overline{A}_1, \dots, \overline{A}_m$ are also simple algebras and, as a consequence, they  are all minimal ideals of $\overline{A}$. Thus it follows that $k=m$ and for all $1\leq i \leq k$ there exists $1\leq j\leq k$ such that $\pi(B_i)=\overline{A}_j$. Moreover, since $B_1, \dots, B_k$ are simple algebras, $\pi_{|B_i}$ is an isomorphism of ordinary algebras for all $1\leq i \leq k$, i.e., for all $1\leq i \leq k$ there exists $1\leq j\leq k$ such that
	$B_i \cong \overline{A}_j .$
\end{remark}

Now we are in position to prove the Gordienko-Kochetov's conjecture.

\begin{theorem}\label{Theorem on Gordienko-kochetov's conjecture}
	Let $L$ be a Lie algebra over a field $F$ of characteristic zero. If $A$ is a finite dimensional $L$-algebra over $F$, then $\exp^L(A)=\exp(A)$.
\end{theorem}
\begin{proof}
	Clearly by definition of $\exp(A)$ and $\exp^L(A)$ and by relation \eqref{relazione codimensioni} it follows that $\exp(A)\leq\exp^L(A)$.

	In order to prove the opposite inclusion, let us suppose that $\exp^L(A)=d$ and consider the canonical projection $\pi:A \to A/J=\overline{A}$ where $J=J(A)$ is the Jacobson radical of $A$. If $\overline{A}=\overline{A}_1\oplus \cdots \oplus \overline{A}_m$ with $\overline{A}_1, \dots, \overline{A}_m$  $L$-simple algebras, by Theorem \ref{L-exp}  $$d=\dim_F(\overline{A}_{i_1}\oplus \overline{A}_{i_2}\oplus\cdots\oplus \overline{A}_{i_r})$$ for some $L$-subalgebra $\overline{A}_{i_1}\oplus \overline{A}_{i_2}\oplus\cdots\oplus \overline{A}_{i_r}$ of $\overline{A}$ such that $1\leq r \leq m,$ $ i_s\neq i_p,$ $ 1\leq s,p\leq r$, and
	$$A_{i_1}^L A^+  A_{i_2}^L A^+ \cdots A^+  A_{i_r}^L\neq 0,$$
	where $A^+=A+ F\cdot 1$ and $A_{i_1}, \dots ,A_{i_r}$ are subalgebras of $A$ (not necessary $L$-invariant) such that $\pi(A_{i_s}) =\overline{A}_{i_s}$ for all $1\leq s \leq r$. Now let consider a maximal semisimple subalgebra $B$ of $A$ such that $A=B+J$. If we write	
	$
	B=B_1 \oplus\cdots \oplus  B_k
	$
	with $B_1, \dots, B_k$ simple algebras (not necessary $L$-invariant), then by Remark \ref{rmk isomorphism B_i and A_i} $k=m$ and for all $1\leq s \leq m$ there exists $1\leq j_s \leq m$  such that $A_{i_s}=B_{j_s}$. Since for all $1\leq s \leq r$, $B_{j_s}\cong \overline{A}_{i_s}$  (as ordinary algebras), it follows that 
	$$d=\dim_F(B_{j_1}\oplus B_{j_2}\oplus\cdots\oplus B_{j_r})$$
	with $ j_p\neq j_s$ for all  $ 1\leq p,s\leq r$ and $B_{j_1}^L A^+  B_{j_2}^L A^+ \cdots A^+  B_{j_r}^L\neq 0$.
Then by Lemma \ref{A_1J...JA_k} it follows that $B_{j_1}J B_{j_2}J\dots J B_{j_r}\neq 0$ and by Theorem \ref{exp} we are done.
\end{proof}

\section{On upper triangular matrix algebras with derivations}

In this section we collect some results concerning the upper triangular matrix algebras with derivations.

For $n>1$, let $UT_n$ be the algebra of $n \times n$ upper triangular matrices over $F$. In \cite{CoelhoPolcinoMilies1993} the authors studied the derivations of $UT_n$ and proved that
	any derivation of $UT_n$ is inner.

Let us consider the algebra $UT_2$ where $L$ acts trivially on it. Since $x^\gamma \equiv 0$ for all $\gamma \in L$ is a differential identity of $UT_2$, we are dealing with ordinary identities. Thus by \cite{Malcev1971} we have the following result.

\begin{theorem}
	\label{ThmIdCnOrdinaryUT2}
	\
	\begin{itemize}
		\item[1)] $\I^L(UT_{2})=\langle [x,y][z,w]  \rangle_{T_L}$.
		
		\vspace{1mm}
		\item[2)]  $\{x_{i_{1}}\dots x_{i_{m}}[x_{k},x_{j_{1}},\dots,x_{j_{n-m-1}}] \  | \ i_{1}<\dots<i_{m}, \ k>j_{1}<\dots<j_{n-m-1}, \ m\neq n-1 \}$ is a basis of $P_{n}^L$ modulo $P_{n}^L\cap \I^L(UT_{2})$.
		
			\vspace{1mm}
		\item[3)]  $c^L _{n}(UT_{2})=2^{n-1}(n-2)+2.$

	\end{itemize}
\end{theorem}

In \cite{GiambrunoRizzo2018}, Giambruno and Rizzo introduced another algebra with derivations generating a variety of almost polynomial growth. They considered $UT_2^\varepsilon$ to be the $L$-algebra $UT_2$ where $L$ acts on it as the 1-dimensional Lie algebra spanned by the inner derivation $\varepsilon=\ad_{e_{22}}$, where $e_{22}$ is the matrix unit whose non-zero entry is $1_F$ in position $(2,2)$. The authors proved the following.

\begin{theorem}\label{UT2e}{\cite[Theorem 5]{GiambrunoRizzo2018}}
	\begin{itemize}
		\item[1)] $\I^L(UT_2^\varepsilon)=\langle x^{\varepsilon^2}-x^{\varepsilon}, \  x^{\varepsilon}y^{\varepsilon}, \ [x,y]^\varepsilon - [x,y]\rangle_{T_L}.$
		
		\vspace{1mm}
		\item[2)] $\{ x_{i_{1}}\dots x_{i_{m}}[x_{k},x_{j_{1}},\dots,x_{j_{n-m-1}}],\
		x_{h_{1}}\dots x_{h_{n-1}}x_{r}^{\varepsilon},\ x_{i_{1}}\dots x_{i_{m}}[x_{l_{1}}^{\varepsilon},x_{l_{2}},\dots,x_{l_{n-m}}]\ | \ i_{1}<\dots<i_{m}, \ k>j_{1}<\dots<j_{n-m-1}, \ h_{1}<\dots<h_{n-1}, \ l_{1}<\dots<l_{n-m}, \ m\neq n-1\}$ is a basis of $P_{n}^L$ modulo $P_{n}^L\cap \I^L(UT_{2}^\varepsilon)$.
		
		 	\vspace{1mm}
		 \item[3)] $c_n^L(UT_2^\varepsilon)=2^{n-1}n-1.$
	\end{itemize}
	
\end{theorem}

Notice that from the above theorems it follows that $\I^L(UT_2)\nsubseteq \I^L(UT_2^\varepsilon)$ and $\I^L(UT_2^\varepsilon)\nsubseteq \I^L(UT_2)$ for any Lie algebra $L$ that  acts as $F\varepsilon$ on $UT_2^\varepsilon$ and as the zero Lie algebra on $UT_2$. Thus by \cite{Kemer1979} and \cite[Theorem 15]{GiambrunoRizzo2018} we have the following.

\begin{theorem}
		\label{Thm:almost polynomial growth}
		The algebras $UT_2$ and $UT_2^\varepsilon$ generate two distinct varieties of algebras with derivations of almost polynomial growth.
\end{theorem}

Now denote by $UT_2 ^\eta$ the $L$-algebra $UT_2$ where $L$ acts on it as the 1-dimensional Lie algebra spanned by a derivation $\eta$ of $UT_2$.
Notice that since any derivation of $ UT_2 $ is inner, $\Der(UT_2)$ is the 2-dimensional metabelian Lie algebra with basis $ \{\varepsilon,\delta\} $, where $\varepsilon=\ad_{e_{22}}$ and $\delta=\ad_{e_{12}}$, where $e_{12}$ is the matrix unit whose non-zero entry is $1_F$ in position $(1,2)$. Then $\eta =\alpha \varepsilon + \beta \delta,$ for some $\alpha, \beta \in F$.
In \cite{Rizzo2021} the author proved the following.

\begin{theorem}\cite[Theorem 12]{Rizzo2021}\label{UT2^eta} Let  $\eta =\alpha \varepsilon + \beta \delta \in \Der(UT_2)$ such that $\alpha, \beta \in F$ are not both zero.
	\begin{itemize}
		\item[1)] If $\alpha\neq 0$, then $\I^{L}(UT_2^{\eta})=\langle x^{\eta^{2}}-\alpha x^{\eta}, \ x^{\eta}y^{\eta},\ [x,y]^{\eta}-\alpha [x,y] \rangle_{T_{L}}$. Otherwise, 
		$\I^{L}(UT_2^{\eta})=\langle x^{\eta^2}, \  x^{\eta}y^{\eta}, \ [x,y]^\eta, \ [x,y][z,w], \ x^\eta[y,z]\rangle_{T_L}$.
		
		\vspace{1mm}
		\item[2)] $c_n^{L}(UT_2^{\eta})=2^{n-1}n+1$.
	\end{itemize}
	
\end{theorem}

As a consequence we get the following corollary.

\begin{corollary}
	\label{Cor UT_2^eta}
	 Let  $\eta =\alpha \varepsilon + \beta \delta \in \Der(UT_2)$ for some $\alpha,\beta\in F$.
	If $\alpha\neq 0$, then $\V^{L}(UT_2^{\eta})= \V^L(UT_2^\varepsilon)$. Otherwise, 
	$UT_2\in \V^{L}(UT_2^{\eta})$.
\end{corollary}	
\begin{proof}
	Let suppose first that $\alpha \neq 0$. Notice that for any  $\beta \in F$, $\I^L(UT_2^\eta)=\I^L(UT_2^{\alpha\varepsilon})$ where $UT_2^{\alpha\varepsilon}$ is the $L$-algebra $UT_2$ where $L$ acts on it as the 1-dimensional Lie algebra spanned by the derivation $\alpha\varepsilon=\alpha \ad_{e_{22}}$.
	On the other hand it is clear that $\V^L(UT_2^\varepsilon)=\V^L(UT_2^{\alpha \varepsilon})$ and then $\V^L(UT_2^\varepsilon)=\V^L(UT_2^\eta)$.
	
	Suppose now that $\alpha=0$. If $\beta=0$ there is nothing to prove, so let $\beta \neq 0$. Notice that we can regard $UT_2$ as an algebra with $F\eta$-action by derivation where $\eta$ acts trivially on $UT_2$, i.e., $x^{\eta}\equiv 0$ is differential identity of $UT_2$. Then by Theorem \ref{UT2^eta} it follows that $UT_2\in \V^L(UT_2^\eta)$.
\end{proof}

\begin{remark}\label{rmk UT_2^D}
	Let us denote by $ UT_2^D$ the $L$-algebra $UT_2$ where $L$ acts on it as the all Lie algebra $\Der(UT_2)$. In \cite{GiambrunoRizzo2018} the authors proved 
	that $UT_2^\varepsilon \in \V^L(UT_2^D)$ and, as a consequence, $\V^L(UT_2^D)$ has no almost polynomial growth. 
\end{remark}

\begin{proposition}
	\label{prop UT_n}
	If a Lie algebra $L$ acts on $UT_n$, $n \geq 2$, by derivations, then either $UT_2 \in \V^L(UT_n)$ or $UT_2^\varepsilon \in \V^L(UT_n)$.
\end{proposition}

\begin{proof}
	Suppose first that $n=2$. If $L$ acts trivially there is nothing to prove, so let $L$ acts non-trivially on $UT_2$. If $L$ acts as a 1-dimensional Lie subalgebra of $\Der(UT_2)$,  then by Corollary \ref{Cor UT_2^eta}  we are done. Then let us assume that $L$ acts as a 2-dimensional Lie subalgebra of $\Der(UT_2)$. Then $L$ acts as the all Lie algebra $\Der(UT_2)$ since $\dim_F \Der (UT_2)=2$. Thus by Remark \ref{rmk UT_2^D} we are done also in this case.
	
	Suppose then that $n > 2$
	and 
consider the ideal $I=\spn_F\{e_{ij}\ : \ i<j, \ j\neq 2\}$ of $UT_n$ where $e_{ij}$'s are the usual matrix units. Since any derivation of $UT_n$ is inner (see \cite{CoelhoPolcinoMilies1993}), by the multiplication table of $UT_n$ it follows that $I$ is an $L$-ideal of $UT_n$ and $A=UT_n/I$ is an $L$-algebra. 
If $B=\spn_F\{e_{11}+I, e_{22}+I, e_{12}+I\}$, then $B$ is a $L$-subalgebra of $A$. Moreover, $B$ is isomorphic to $UT_2$ as ordinary algebras. Thus as $L$-algebra $B$ is isomorphic to $UT_2$ with $L$-action and by the first part of the proof we have that  either $UT_2 \in \V^L(B)$ or $UT_2^\varepsilon \in \V^L(B)$. Since $\V^L(B)\subseteq\V^L(A)\subseteq\V^L(UT_n)$ the proof is complete.
\end{proof}

Notice that as a consequence of the Theorem \ref{Thm:almost polynomial growth} and Proposition \ref{prop UT_n} we have that $UT_2$ and $UT_2^\varepsilon$ are the only upper triangular matrix algebras generating varieties of algebras with derivation of almost polynomial growth.

\section{Differential varieties of almost polynomial growth}
In this section 
we shall characterize the varieties of algebras with derivations of almost polynomial growth 
in case $L$ is a finite dimensional solvable Lie algebra. 

\smallskip

Let $V$ be a finite dimensional vector space over $F$. The space $\mathfrak{gl}(V)$ of all linear maps from $V$ to $V$ is a Lie algebra, if we define the Lie bracket $[-,-]$ by
$$[v, w] := v \circ w - v \circ w$$
for all $ v, w\in \mathfrak{gl}(V ),$
where $\circ$ denotes the composition of maps. Then we have the following property for solvable Lie subalgebra of $\mathfrak{gl} (V)$.

\begin{theorem}[\cite{Humphreys1972}, Theorem 4.1]
	\label{Thm:eigenvector solvable Lie algebra}
	Let $V$ be a finite dimensional non-zero vector space over an algebraically closed field $F$ of characteristic zero. Suppose that $L$ is a finite dimensional solvable Lie subalgebra of $\mathfrak{gl} (V)$. Then there is some non-zero $v\in V$ which is a simultaneous eigenvector for all $\delta \in L$, i.e., $v^\delta=\alpha_\delta v$ with $\alpha_\delta \in F$ for all $\delta \in L.$
\end{theorem}

Next lemmas will be useful to establish a structural result about $L$-varieties of polynomial growth.

\begin{lemma}
	\label{prop AiJAk in J^q}
Let $A=A_1\oplus A_2 + J$ be a finite dimensional algebra over an algebraically closed field $F$ of characteristic zero,  where   $A_1 \cong A_2 \cong F$,  $A_1 J A_2 \neq 0$ and $A_1 J A_2 \subseteq J^q$, with $q$ such that $J^q \neq 0$ and $J^{q+1}=0$. 
If $L$ is a finite dimensional solvable Lie algebra acting on $A$ by derivation, then either $UT_2  \in \V^L(A)$ or $UT_2^\varepsilon \in \V^L(A)$.

\end{lemma}
\begin{proof}
	Since $A_1 J A_2\neq 0$, if $e_1$ and $e_2$ denote the unit elements of $A_1$ and $A_2$, respectively, we have that $e_1 Je_2\neq 0$ with $e_1 e_2=e_2 e_1=0$. Moreover, since $A_1JA_2 \subseteq J^q$, there exists $j\in J^q$ such that $e_1je_2\neq 0$.
	Let $B=\spn_F \{e_1 j^d e_2 \ | \ d \in \mathcal{B}_{U(L)}\}$. Clearly $B$ is a subalgebra of $A$. Moreover, by the  Poincaré--Birkhoff--Witt Theorem and the Leibniz rule we have that, for all $h\in \mathcal{B}_{U(L)}$,
	$$(e_1 j^d e_2)^h=e_1^h j^d e_2 + e_1 j^{dh} e_2 + e_1 j^d e_2^h + \sum e_1^{h'} j^{dh''} e_2^{h'''}$$
	where $h', h'', h'''$ are suitable elements of $U(L)$ such that at least two between them are not in $\spn_F \{1_{U(L)}\}$. Since $A_1 \cong A_2 \cong F$, by Lemma \ref{prop B_i^L} it follows that $e_1^h, e_2^h\in J$ for all $h\in U(L)$. Then since $j\in J^q$ , we have that $(e_1 j^d e_2)^h=e_1 j^{dh}e_2$ for all $h\in U(L)$. Thus it follows that $B$ is an $L$-subalgebra of $A$. Now, since $L$ is solvable, by Theorem \ref{Thm:eigenvector solvable Lie algebra} there exists a non-zero $b\in B$ which is a simultaneous eigenvector
	for all $\delta \in L$. Thus, as a consequence of the Poincaré--Birkhoff--Witt Theorem, we have that $b^h=\alpha_h b$, $\alpha_h\in F$, for all $h\in U(L)$. Notice also that $b=e_1 b e_2$ and $\bar{j}b=b\bar{j}=0$ for all $\bar{j}\in J$.
	Let now $C$ be the subalgebra of $A$ generating by $\{e_1^h, e_2^h,b \ | \ h\in \mathcal{B}_{U(L)}\}$. Clearly $C$ is $L$-invariant and its vector subspace $D=\spn_F\{e_1, e_2, b\}$ is a subalgebra (not necessarily $L$-invariant) such that  $b^h=\alpha_h b$, $\alpha_h\in F$, for all $h\in U(L)$ and $D\cong UT_2$ as ordinary algebras. We shall prove that either $UT_2\in \V^L(C)$ or $UT_2^\varepsilon\in \V^L(C)$ and since $\V^L(C)\subseteq\V^L(A)$ the theorem will be proved.
	
	Suppose first that $b^h=0$ for all $h\in\mathcal{B}_{U'(L)}$, where $\mathcal{B}_{U'(L)}$ is a basis of the non-unital enveloping algebra $U'(L)$ of $L$, and let $f\in \I^L(C)$ be a multilinear $L$-polynomial of degree $n$. By Theorem \ref{ThmIdCnOrdinaryUT2}, $f$ can be written as
	\[ f=\alpha x_1 \dots x_n + \sum_{\mathcal{I}, \mathcal{J}} \alpha_{\mathcal{I}, \mathcal{J}} x_{i_1}\dots x_{i_m} [x_k, x_{j_1}, \dots , x_{j_{n-m-1}}]+ g, \]
	where $g\in \I^L(UT_2)$, $\mathcal{I}=\{i_1, \dots, i_m\}$ and $\mathcal{J}= \{j_1, \dots, j_{n-m-1}\}$ with $i_1<\dots< i_m$, $k>j_1< \dots< j_{n-m-1}$ and $0\leq m \leq n-2$. First of all notice that if $\varphi$ is an evaluation that send at least one variables in $b$ and all the others in $\{e_1, e_2\}$, then $\varphi(g)=0$: indeed, since $b^h=0$,  $e_1^h, e_2^h\in J$ for all $h\in \mathcal{B}_{U'(L)}$ and $\bar{j}b=b\bar{j}=0$ for all $\bar{j}\in J$, all the monomials of $g$ in which appear at least an element of $U'(L)$ as exponent of some variables are evaluated in zero. So it is not restrictive to assume that $g$ is a linear combination of monomials in which all the variables have as exponent $1_{U(L)}$, i.e., $g\in P_n\cap\I^L(UT_2)$. Thus since $\spn_F\{e_1, e_2, b\}\cong UT_2$ as ordinary algebras, it follows that $\varphi(g)=0$.
	
	For fixed $\mathcal{I}$ and $\mathcal{J}$ the evaluation $x_{i_1}=\dots =x_{i_m}=e_1+e_2$, $x_k=b$, $x_{j_1}=\dots = x_{j_{n-m-1}}=e_{2}$ gives $\alpha_{\mathcal{I}, \mathcal{J}}=0$ since  $e_1^h, e_2^h\in J$ for all $h\in \mathcal{B}_{U'(L)}$, $\bar{j}b=b\bar{j}=0$ for all $\bar{j}\in J$ and $\spn_F\{e_1, e_2, b\}$ has the same multiplication table of $UT_2$. Moreover, by choosing $x_1=\dots=x_{n-1}=e_{1}$ and $x_n=b$ we get $\alpha=0$. Thus $f=g\in \I^L(UT_2)$ and so $UT_2\in\V^L(C)$.
	
	Let now suppose that there exists $h\in\mathcal{B}_{U'(L)}$ such that $b^h=\alpha_h b$ with $\alpha_h\in F$, $\alpha_h\neq 0$. Clearly it is not restrictive suppose that $h\in L$. Moreover, we may assume that $b^h=b$ and $b^{h'}=0$ for all $h'\in \mathcal{B}_{U'(L)}$ such that $h'\neq h$. In fact, if not, we consider a new basis of $U'(L)$ obtained from our basis $\mathcal{B}_{U'(L)}$ by substituting any element $h'\in \mathcal{B}_{U'(L)}$, $h'\neq h$, such that $b^{h'}=\alpha_{h'}b$, $\alpha_{h'}\neq 0$, with the element $h'-\alpha_{h'}\alpha_h^{-1}h$ and $h$ with $\alpha_h^{-1}h$. 
	
	Let $f\in \I^L(C)$ be a multilinear $L$-polynomial of degree $n$. By Theorem \ref{UT2e}, $f$ can be written as
	\begin{align*}
		f=&\alpha x_1 \dots x_n + \sum_{\mathcal{I}, \mathcal{J}} \alpha_{\mathcal{I}, \mathcal{J}} \ x_{i_1}\dots x_{i_m} [x_k, x_{j_1}, \dots , x_{j_{n-m-1}}]+ \sum_{r=1}^{n} \beta_{r} \ x_{q_1}\dots x_{q_{n-1}}x_r^{h}\\
		&+ \sum_{\mathcal{P}}\gamma_{\mathcal{P}} \ x_{p_1}\dots x_{p_s}[x_{t_1}^h, x_{t_2}, \dots, x_{t_{n-s}}] + g
	\end{align*}
	where $g\in\I^L(UT_2^\varepsilon)$, $\mathcal{I}=\{i_1, \dots, i_m\}$, $\mathcal{J}= \{j_1, \dots, j_{n-m-1}\}$ and $\mathcal{P}=\{p_1, \dots, p_s\}$ with $i_1<\dots< i_m$, $k>j_1< \dots< j_{n-m-1}$, $p_1<\dots<p_s$, $t_1<\dots<t_{n-s}$, $0\leq m, s \leq n-2$ and $q_1<\dots<q_{n-1}$. 
	
	Notice that $\varphi(g)=0$ for all evaluation $\varphi$ that send at least one variables in $b$ and all the others in $\{e_1, e_2\}$. In fact, if $\varphi$ is such evaluation, then every monomials in $g$ with at least two elements of $U'(L)$ as exponent of different variables and that one with an element $h'\in U'(L)$, $h'\neq h$, as exponent of some variables are evaluated in zero. So, we may assume that $g$ is a linear combination of monomials in which at most one varialbe has as exponent $h$ and all the others $1_{U(L)}$. Thus since $g\in\I^L(UT_2^\varepsilon)$, $\spn_F\{e_1, e_2, b\}\cong UT_2$ as ordinary algebras and $b^h=b$, it follows that $\varphi(g)=0$.
	
	Suppose first $n=1$. Then $f=f(x)=\alpha x + \beta x^h+ g(x)$ where $g(x)\in\I^L(UT_2^\varepsilon)$. If we evaluate $x=e_1$, we obtain $\alpha e_1+ \beta e_1^h +\bar{g}(e_1)=0$, where $\beta e_1^h +\bar{g}(e_1)\in J$ since $e_1^d\in J$ for all $d\in U'(L)$. Thus since $A_1 \cap J=0$, it follows that $\alpha=0$. Now by making the substitution $x=b$, we get $\beta=0$ because $b^{h'}=0$ for all $h'\in \mathcal{B}_{U'(L)}$, $h'\neq h$. Thus for $n=1$, $f=g\in \I^L(UT_2^\varepsilon)$.
	
	Now if $n=2$, then $f=f(x,y)=\alpha_1 x y+ \alpha_2 yx+ \beta_1 y x^h +\beta_2 x y^h + \gamma x^h y+g(x,y)$ where $g(x,y)\in\I^L(UT_2^\varepsilon)$. By making the evaluation $x=y=e_1$ we get $(\alpha_1+\alpha_2)e_1+ (\beta_1+\beta_2)e_1 e_1^h + \gamma e_1^h e_1+ \bar{g}(e_1,e_1)=0$ where $ (\beta_1+\beta_2)e_1 e_1^h + \gamma e_1^h e_1+ \bar{g}(e_1,e_1)\in J$ since  $e_1^d\in J$ for all $d\in U'(L)$. Thus it follows that $\alpha_1+\alpha_2=0$ since $A_1\cap J=0$. Now by choosing $x=e_2$ and $y=b$ we have that $\alpha_2=0$ and $\alpha_1=-\alpha_2=0$. Now by evaluating $x=b$ and $y=e_1$ we get $\beta_1=0$. From the substitution $x=e_1$ and $y=b$ we obtain $\beta_2=0$. Finally, the evaluation $x=b$ and $y=e_2$ gives $\gamma=0$. Thus also in this case we proved that $f=g\in \I^L(UT_2^\varepsilon)$.
	
	Let now $n\geq 3$. For fixed $\mathcal{I}$ and $\mathcal{J}$ consider the evaluation $x_{i_1}=\dots =x_{i_m}=e_1+e_2$, $x_k=b$ and $x_{j_1}=\dots=x_{j_{n-m-1}}=e_2$. Since $|\mathcal{J}|\geq 2$ and $k>j_{1}$, then $\alpha_{\mathcal{I}, \mathcal{J}}=0$ because $e_1^h, e_2^h\in J$, $\bar{j}b=b\bar{j}=0$ for all $\bar{j}\in J$ and $\spn_F\{e_1, e_2, b\}$ has the same multiplication table of $UT_2$. Moreover, for fixed $\mathcal{P}$ by making the evaluation $x_{p_1}=\dots=x_{p_s}=e_1+e_2$, $x_{t_1}=b$ and $x_{t_2}=\dots= x_{t_{n-s}}=e_2$, we get $\gamma_{\mathcal{P}}=0$ if $\mathcal{P}\neq \{1, \dots, s\}$. Thus we may assume that
	\[ f=\alpha x_1 \dots x_n + \sum_{r=1}^{n} \beta_{r}\ x_{q_1}\dots x_{q_{n-1}}x_r^{h}+ \sum_{s=0}^{n-2}\gamma_{s} \ x_{1}\dots x_{s}[x_{s+1}^h, x_{s+2}, \dots, x_{t_{n}}] + g, \]
	where $g\in\I^L(UT_2^\varepsilon)$ and $q_1<\dots<q_{n-1}$. If $n$ is odd, then from the evaluation $x_1=b$, $x_2=\dots=x_n=e_2$ we get $\alpha=0$. Also the evaluation $x_1=e_1$, $x_2=b$, $x_3=\dots=x_n=e_2$ gives $\gamma_1=0$ since $n\geq 3$. On the other hand, if $n$ is even, then by making the evaluation $x_1=e_1$, $x_2=b$, $x_3=\dots=x_n=e_2$, we obtain $\alpha=0$ since $n\geq 4$. Moreover, by choosing $x_1=b$, $x_2= \dots =x_n=e_2$ we get $\gamma_0=0$.  Thus we may assume that $\alpha=\gamma_0=\gamma_1=0$ for all $n\geq 3$. Now if for all $2\leq s\le n-2$ we consider the evaluation $x_1=\dots=x_s=e_1+e_2$, $x_{s+1}=b$, $x_{s+2}=\dots=x_n=e_2$, then we get $\gamma_s=0$.  Finally, for $1\leq r \leq n$ by choosing $x_r=b$ and $x_{q_1}=\dots=x_{q_{n-1}}=e_1$ we obtain $\beta_r=0$. As a consequence we have that for all $n\geq 3$ $f=g\in \I^L (UT_2^\varepsilon)$. Thus we prove that  $UT_2^\varepsilon\in \V^L(C)$ and the proof is complete.
\end{proof}

A basic result we shall need in what follows is the Lie's theorem for solvable Lie algebras.

\begin{theorem}[\cite{Humphreys1972}, Lie's Theorem]
	\label{Lie's Thm}
	Let $V$ be a finite dimensional non-zero vector space over an algebraically closed field $F$ of characteristic zero and let $L$ be a finite dimensional solvable Lie subalgebra of $\mathfrak{gl} (V)$. Then there is a basis of $V$ in which every element of $L$ is represented by an upper triangular matrix.
\end{theorem}

Recall that if $V$ is a $p$-dimensional vector space over $F$ and we fix a basis of it, then we may identify $\mathfrak{gl}(V)$ with the set of all $p \times p$ matrices over $F$, and we write $\mathfrak{gl}_p$ for the Lie algebra of all $p \times p$ matrices over $F$ with the Lie bracket defined by $[a,b]=ab-ba$, for all $a,b\in \mathfrak{gl}_p$.

\begin{theorem}
	\label{Thm: Characterization in terms of $UT_2^epsilon$}
	Let $L$ be a finite dimensional solvable Lie algebra over a field $F$ of characteristic zero and $A$ be a finite dimensional
	$L$-algebra over $F$. Then the sequence $c_n^L(A)$, $n\geq 1$, is polynomially bounded if and only if $UT_2, UT_2^\varepsilon \notin \V^L(A)$.
\end{theorem}

\begin{proof}
	First suppose that $c_n^L(A)$ is polynomially bounded. Since, by theorems \ref{ThmIdCnOrdinaryUT2} and \ref{UT2e}, $UT_2$ and $UT_2^\varepsilon$ generate  $L$-varieties of exponential growth, we have $UT_2, UT_2^\varepsilon \notin \V^L(A)$.
	
	Now assume $UT_2, UT_2^\varepsilon \notin \V^L(A)$. Using an argument
	analogous to that used in the ordinary case (see \cite[Theorem
	4.1.9]{GiambrunoZaicev2005book}), we can prove that the differential
	codimensions do not change upon extension of the base field and so
	we may assume $F$ is algebraically closed. 
	Since the Jacobson radical $J=J(A)$ of $A$ is an $L$-ideal, by Proposition \ref{Prop semisimple}
	 we have that $\overline{A}= A/J$ is a semisimple $L$-algebra such that
		$\overline{A}=\overline{A}_1\oplus \cdots \oplus \overline{A}_m$
		where  $\overline{A}_i$ is an $L$-simple algebras such that $\overline{A}_i\cong_L M_{n_i}(F),$ $n_i\geq 1$, for all $1\leq i \leq m$. 
	
	Suppose that $n_i>1$, for some $i$. As a consequence of the Noether-Skolem theorem, all derivations of $M_{n_i}(F)$ are inner (\cite[p.100]{Herstein1968}), then $L$ acts on $M_{n_i}(F)$ as $\ad \overline{L}$ where $\overline{L}$ is a Lie subalgebra of $\mathfrak{gl}_{n_i}$. Since every homomorphic image of a solvable Lie algebra is still solvable, then  $\overline{L}$ is solvable. Thus, by Theorem \ref{Lie's Thm}, $\overline{L}$ is contained in the Lie subalgebra $\mathfrak{U}_{n_i}$ of $\mathfrak{gl}_{n_i}$ of $n_i\times n_i$ upper triangular matrices over $F$. Hence it follows that $UT_{n_i}$ is an $L$-invariant subalgebra of $M_{n_i}(F)$, i.e., $UT_{n_i}\in \V^L(M_{n_i}(F))\subseteq \V^L(A)$ and by Proposition \ref{prop UT_n} we reach a contradiction.	Thus $\overline{A}_i\cong F$ for all $1\leq i \leq  m$.
	
	Notice that if we consider the Wedderburn-Malcev decomposition of $A$ as ordinary algebra,
	$A=A_1\oplus\dots\oplus A_m+ J,$ 
	then   
	by Remark \ref{rmk isomorphism B_i and A_i}
	$A_i\cong F$ (as ordinary algebras) for all $1\leq i \leq m$. Thus in order to finish the proof, by theorems \ref{Theorem on Gordienko-kochetov's conjecture} and \ref{Characterization of exp}, it is enough to guarantee that $A_i JA_k=0$  for all $1 \leq i,k\leq m$, $i\neq k$.  Suppose to the contrary that there exist $1\leq i,k\leq  m$, $i \neq k$, such that $A_i JA_k\neq 0$. Since $A_i\cong A_k\cong F$, by Lemma \ref{prop B_i^L} $A_i^L\subseteq J$ and $A_k^L\subseteq J$; then $B=A_i \oplus A_k +J$ is a $L$-subalgebra of $A$. We claim that either $UT_2\in \V^L(B)$ or $UT_2^\varepsilon\in \V^L(B)$. Let $q$ the largest integer such that $J^q\neq 0$ and $J^{q+1}=0$. We shall prove the claim by induction on $q$. If $q=1$, then since $A_i JA_k\subseteq J$, by Lemma \ref{prop AiJAk in J^q} we are done. So, let us assume that $q>1$. If $A_i JA_k\subseteq J^q$, then by Lemma \ref{prop AiJAk in J^q} we are done also in this case. So, let us suppose that $A_i J A_k\nsubseteq J^q$. Then since $J^q$ is an $L$-ideal of $B$, $\overline{B}=B/J^q$ is an $L$-algebra such that $\overline{B}= \overline{A}_i \oplus  \overline{A}_k +\overline{J}$ where $\overline{A}_i \cong  \overline{A}_k\cong F$, $\overline{J}^{q}=0$ and $\overline{A}_i \overline{J} \ \overline{A}_k\neq 0$. Then by the inductive hypothesis we have that  either $UT_2\in \V^L(\overline{B})$ or $UT_2^\varepsilon\in \V^L(\overline{B})$ and since $\V^L(\overline{B}) \subseteq \V^L (B)$ the claim is proved.
	 Since $B$ is a $L$-subalgebra of $A$, we have proved that either $UT_2\in\V^L(A)$ or $UT_2^\varepsilon\in\V^L(A)$, a contradiction, and the theorem is proved.		
\end{proof}

As a consequence we have the following.

\begin{corollary}
	If $L$ is a finite dimensional solvable Lie algebra, then the algebras $UT_{2}$ and $UT_{2}^{\varepsilon}$ are the only finite dimensional $L$-algebras generating varieties of almost polynomial growth.
\end{corollary}


\begin{thebibliography}{99}
	

	
	\bibitem{CoelhoPolcinoMilies1993} S.P. Coelho, C. Polcino Milies, {\em Derivations of upper triangular matrix rings}, Linear Algebra Appl. {\bf 187} (1993), 263--267.
	
	\bibitem{GiambrunoIoppoloLaMattina2016} A. Giambruno, A. Ioppolo, D. La Mattina, {\em Varieties of algebras with superinvolution of almost polynomial growth}, Algebr. Represent. Theory {\bf 19} (2016), no. 3, 599--611.
	
	\bibitem{GiambMish} A. Giambruno, S. Mishchenko, {\em On star-varieties with almost polynomial growth}, Algebra Colloq. {\bf 8} (2001), 33--42.
	
	
	\bibitem{GiambMish2001CA}A. Giambruno, S. Mishchenko, {\em Growth of the $*$-codimensions and Young diagrams}, Comm. Algebra Comm. Algebra {\bf29} (2001), no. 1 277--284.
	
		
	\bibitem{GiambMishZai} A. Giambruno, S. Mishchenko, M. Zaicev, {\em Polynomial identities on superalgebras and almost polynomial growth}, Comm. Algebra {\bf 29} (2001), no. 9, 3787--3800 (Special issue dedicated to Alexei Ivanovich Kostrikin).
	
	\bibitem{GiambrunoRizzo2018} A. Giambruno, C. Rizzo, {\em Differential identities, $2\times 2$ upper triangular matrices and varieties of almost polynomial growth}, J. Pure Appl. Algebra {\bf 223} (2019), no. 4, 1710--1727.
	
	\bibitem{GiambrunoZaicev1998} A. Giambruno, M. Zaicev, {\em On codimension growth of finitely generated associative algebras}, Adv. Math. {\bf 140} (1998), 145--155.
	
	\bibitem{GiambrunoZaicev1999} A. Giambruno, M. Zaicev, {\em Exponential codimension growth of PI algebras: an exact estimate}, Adv. Math. {\bf 142} (1999), 221--243.
	
	
	\bibitem{GiambrunoZaicev2005book} A. Giambruno, M. Zaicev, Polynomial identities and asymptotic methods, Math. Surv. Monogr., AMS, Providence, RI, {\bf 122} (2005).
	
	\bibitem{Gordienko2013JA} A.S. Gordienko, {\em Asymptotics of H-identities for associative algebras with an H-invariant radical}, J. Algebra {\bf 393} (2013), 92--101.
	
	\bibitem{GordienkoKochetov2014} A.S. Gordienko, M.V. Kochetov, {\em Derivations, gradings, actions of algebraic groups, and codimension growth of polynomial identities}, Algebr. Represent. Th. {\bf 17} (2014), no. 2, 539--563.
	
	\bibitem{Herstein1968} I.N. Herstein, {\em Noncommutative rings}. The Carus Mathematical Monographs, No. 15 Mathematical Association of America; distributed by John Wiley \& Sons, Inc., New York, 1968.
	
	\bibitem{Hoch} G. Hochschild. {\em Semi-simple algebras and generalized derivations}, Amer. J. Math. {\bf 64} (1942), no. 1, 677--694.
	

	 
	\bibitem{Humphreys1972} J.E. Humphreys, Introduction to Lie Algebras and Representation Theory, Springer-Verlag, New York, 1972.
	
	\bibitem{IoppoloMartino2018} A. Ioppolo, F. Martino, {\em Varieties of algebras with pseudoinvolution and polynomial growth}, Linear Multilinear Algebra {\bf 66} (2018), no. 11, 2286--2304.
	
	
	\bibitem{Kemer1979} A.R. Kemer, {\em Varieties of finite rank}, Proc. 15-th All the Union Algebraic Conf., Krasnoyarsk. {\bf 2} (1979), p. 73 (in Russian).
	
	 \bibitem{Kharchenko1978}   V.K. Kharchenko, {\em Differential identities of prime rings}, Algebra and Logic {\bf 17} (1978), no. 2, 155--168.
	

	\bibitem{Kharchenko1979} V.K. Kharchenko, {\em Differential identities of semiprime rings}, Algebra Logic {\bf 18} (1979), 86--119.
	
	
	\bibitem{Malcev1971} J.N. Malcev, {\em A basis for the identities of the algebra of upper triangular matrices}, Algebra i Logika {\bf 10} (1971), 393--400.

\bibitem{MartinoRizzo2022}  F. Martino, C. Rizzo, {\em Differential identities and and varieties of almost polynomial growth},  Israel J. Math (2022), in Press.	

	
	\bibitem{Regev1972} A. Regev, {\em Existence of identities in $A\otimes B$}, Israel J. Math. {\bf 11} (1972), 131--152.
	
	\bibitem{Rizzo2020ART} C. Rizzo, {\em The Grassmann algebra and its differential identities}, Algebr. Represent. Theory {\bf 23} (2020), no. 1, 125--134.
	
	
	\bibitem{Rizzo2021} C. Rizzo, {\em Growth of differential identities},  Polynomial Identities in Algebras, Springer INdAM Series \textbf{44} (2021), 383--400.
	
	\bibitem{RizzodosSantosVieira2021} C. Rizzo, R.B. dos Santos, A.C. Vieira, {\em Differential identities and polynomial growth of the codimensions}, Algebr. Represent. Theory (2022), https://doi.org/10.1007/s10468-022-10163-0.


	\bibitem{Valenti2011} A. Valenti, {\em Group graded algebras and almost polynomial growth}, J. Algebra \textbf{334} (2011), 247--254.
	

\end{thebibliography}
\end{document}